\documentclass[preprint,12pt]{article}% Use utf-8 encoding for foreign characters
\usepackage{mathrsfs}
\usepackage{amssymb}
\usepackage{amsthm}
\usepackage{graphicx}
\usepackage{authblk}

\usepackage[ruled,lined,linesnumbered,commentsnumbered]{algorithm2e}

\newtheorem{thm}{Theorem}
\newtheorem{lem}[thm]{Lemma}
\newtheorem{cor}[thm]{Corollary}
\newtheorem{prop}[thm]{Proposition}

\newtheorem{problem}[thm]{Problem}

\newtheorem{claim1}{Claim}

\begin{document}

\title{A dichotomy for the kernel by $H$-walks problem in digraphs}

\author{Hortensia Galeana-S\'anchez\thanks{email: hgaleana@matem.unam.mx} \\ {\small Instituto de Matem\'aticas \\ Universidad Nacional Aut\'onoma de M\'exico} \and
C\'esar Hern\'andez-Cruz\thanks{email: cesar@matem.unam.mx (Corresponding Author)} \\ {\small Instituto de Matem\'aticas \\ Universidad Nacional Aut\'onoma de M\'exico} \and
}

\maketitle
\begin{abstract}
Let $H = (V_H, A_H)$ be a digraph which may
contain loops, and let $D = (V_D, A_D)$ be a
loopless digraph with a coloring of its arcs $c:
A_D \to V_H$.   An $H$-walk of $D$ is a walk
$(v_0, \dots, v_n)$ of $D$ such that $(c(v_{i-1},
v_i), c(v_i, v_{i+1}))$ is an arc of $H$, for every
$1 \le i \le n-1$.   For $u, v \in V_D$, we say that
$u$ reaches $v$ by $H$-walks if there exists
an $H$-walk from $u$ to $v$ in $D$.   A subset
$S \subseteq V_D$ is a kernel by $H$-walks of
$D$ if every vertex in $V_D \setminus S$
reaches by $H$-walks some vertex in $S$, and
no vertex in $S$ can reach another vertex in
$S$ by $H$-walks.

A panchromatic pattern is a digraph $H$ such
that every arc-colored digraph $D$ has a kernel
by $H$-walks.   In this work, we prove that every
digraph $H$ is either a panchromatic pattern, or
the problem of determining whether an arc-colored
digraph $D$ has a kernel by $H$-walks is
$NP$-complete.
\end{abstract}

\section{Introduction} \label{sec:Intro}

Sands, Sauer and Woodrow proved in
\cite{SSW} a beautiful result stating that
every digraph whose arcs are colored
with two colors has a kernel by
monochromatic paths.   Since then, the
existence of kernels by monochromatic
paths has been studied both in general
digraphs \cite{GS2,GS3,GS-L-MB} and
in tournaments \cite{GS,HIW,S}.

Linek and Sands generalized in
\cite{LS} the notion of kernel by
monochromatic paths in the following
way.   If $H$ is a digraph, possibly
with loops, we will say that the digraph
$D$ is $H$-arc-colored if the set
of arcs of $D$ has been colored
with the vertices of $H$; we will
usually denote this coloring with $c$.
An {\em $H$-walk} is a walk $(v_0,
\dots, v_n)$ in $D$ such that $(c(v_i,
v_{i+1}), c(v_{i+1}, v_{i+2}))$ is an
arc of $H$; we say that $v_0$ {\em
reaches} $v_n$ {\em by $H$-walks}.
A {\em kernel by $H$-walks} of $D$
is a subset $S$ of $V_D$ such that
it is {\em independent by $H$-walks}
(there are no $H$-walks between
vertices in $S$) and {\em absorbent
by $H$-walks} (for every vertex $u$
in $V_D \setminus S$ there is a vertex
$v$ in $S$ such that $u$ reaches $v$
by $H$-walks).

If $V_H = \{ x, y \}$ and $A_H = \{ (x,x),
(y,y) \}$, then a kernel by $H$-walks
in a digraph $D$ is simply a {\em kernel
by monochromatic paths} (observe
that every monochromatic walk contains
a monochromatic path), and hence,
the aforementioned theorem of Sands,
Sauer and Woodrow states that every
$H$-arc-colored digraph has a kernel
by $H$-walks (for this particular
choice of $H$).   In this context, a very
natural question arises.   Which are
the digraphs $H$ such that every
$H$-arc-colored digraph has a kernel
by $H$-walks?   Arpin and Linek
stated this question in \cite{AL}, and
found some digraphs having this
property, as well as some digraphs
not having this property.   A {\em
panchromatic pattern} is a digraph
$H$ such that every $H$-arc-colored
digraph has a kernel by $H$-walks.
Based on the work of Arpin and Linek,
Galeana-S\'anchez and Strausz
characterized all the panchromatic
patterns in \cite{GS-S}.   Again, a
natural question arose.   How hard
is to determine the existence of a
kernel by $H$-walks if $H$ is not a
panchromatic pattern? For a digraph
$H$, define the {\em kernel by
$H$-walks problem} to be the
decision problem of determining
whether a digraph $D$ has a kernel
by $H$-walks. Considering that $H$
can be any digraph, and the
panchromatic patterns are a very
restricted class, it is natural to think
that there are some digraphs $H$
having a polynomial time solvable
kernel by $H$-walks problem, as
well as some digraphs $H$ having an
$NP$-complete kernel by $H$-walks
problem. It really comes as a surprise
that the former case never happens.
The following theorem is the main
result of this work.

\begin{thm} \label{mainthm}
If $H$ is a digraph, then either $H$ is a
panchromatic pattern (and hence the kernel
by $H$-walks problem is constant time
solvable), or the kernel by $H$-walks
problem is $NP$-complete.
\end{thm}

There are only a handful of articles
dealing with the complexity of the
many variations of the kernel problem.
Chv\'atal proved in 1973 \cite{C} that
the kernel problem is $NP$-complete;
in 1981 \cite{F}, Fraenkel proved that
the kernel problem remained
$NP$-complete even when restricted
to planar graphs with $\Delta^+,
\Delta^- \le 2$ and $\Delta \le 3$; in
2014 \cite{HHC1}, Hell and
Hern\'andez-Cruz proved that the
$3$-kernel problem is $NP$-complete,
even when restricted to digraphs
homomorphic to a (directed) $3$-cycle
with circumference $6$, and, as a
consequence, the kernel problem
remains $NP$-complete even when
restricted to $3$-colorable digraphs.
On the other hand, Bang-Jensen, Guo,
Gutin an Volkmann proved in 1997
\cite{BJGGV} that the kernel problem
is polynomial time solvable for locally
semicomplete digraphs; in \cite{HHC1}
it is proved that the $3$-kernel
problem is polynomial time solvable
for semicomplete multipartite digraphs.

This is the first work dealing with the
complexity of a generalization of the
kernel problem different from the
$3$-kernel problem.   In fact, although
the kernel by monochromatic paths
has been widely studied, its complexity
remained unknown until now.

We refer the reader to \cite{BJD} and \cite{BM}
for general concepts.   Let $D = (V_D, A_D)$ be
a digraph with set of vertices $V_D$ and set of
arcs $A_D$.   A {\em loop} is an arc of the form
$(v,v)$; a vertex $v$ is {\em looped} if $(v,v)$ is
a loop of $D$, and it is {\em loopless} otherwise.
A digraph is {\em looped} ({\em loopless}) if all
its vertices are looped (loopless).

A {\em digon} is an arc $(x,y)$ of $D$ such
that $(y,x)$ is also an arc of $D$; given the
symmetric nature of the definition, we will
sometimes refer to an unordered pair of vertices
$\{x, y\}$ as a digon.   An arc $(x,y)$ of $D$ is
{\em asymmetric} if $(y,x) \notin A_D$. We will 
say that a subset $S$ of $V_D$ is a strong
clique if every pair of vertices in $S$ is a digon.

Given a family $\mathcal{F}$ of digraphs, we
say that a digraph $D$ if {\em $\mathcal{F}$-free}
if no member of $\mathcal{F}$ appears as an
(homomorphic copy of an) induced subdigraph
of $D$.   A property $\mathcal{P}$ of a digraph
$D$ is {\em hereditary} if every induced subdigraph
of $D$ also has the property $\mathcal{P}$.   For
example, the property of being $\mathcal{F}$-free
is a hereditary property.

All walks, paths and cycles are considered
to be directed unless otherwise stated.   The
circumference of a digraph is the length of
its longest cycle, and it is defined to be
zero for acyclic digraphs.   For a positive
integer $k$, a $k$-cycle is a cycle of
length $k$.

The remainder of this work is organized
as follows.   In Section \ref{sec:algorithm}
we show that the kernel by $H$-walks
problem is in the class $NP$, by presenting
a polynomial time algorithm to verify
reachability by $H$-walks.   Section
\ref{sec:minobs} is devoted to obtain a
characterization of panchromatic patterns
in terms of a finite family of forbidden
induced subdigraphs.    In order to prove
$NP$-hardness of the kernel by $H$-walks
problem for every $H$ which is not a
panchromatic pattern, we consider three
cases for $H$; the polynomial reduction
for the most complex case is constructed
in Section \ref{sec:genreduc}.   In Section
\ref{sec:mainres} we deal with the two
remaining cases to complete the proof of
Theorem \ref{mainthm}.   Section
\ref{sec:conclusions} is devoted to present
our concluding remarks and future lines of
work.

\section{A reachability algorithm} \label{sec:algorithm}

In order to prove that the kernel by
$H$-walks problem is in $NP$, it
suffices to show that, given an
instance $(D,H)$ of the problem,
and a subset $S$ of $V_D$, it can
be verified in polynomial time
whether $S$ is a kernel by
$H$-walks of $D$.   It is clear that
this can be achieved through an
algorithm that finds all the vertices
reached by $H$-walks from a given
vertex $v$, in polynomial time.   We
only have to run this algorithm from
every vertex of $D$.

In this section we will propose such
algorithm and prove that it runs in
polynomial time.   The proposed
algorithm is based on BFS.   The
main difference is that, instead of
vertices, our queue will have pairs
$(x,c)$, where $x$ is a vertex and
$c$ is the color of the arc which was
explored in order to reach $x$. Each
of this pairs may join the queue at
most once, but the same vertex
may be considered several times.
This will allow us to know through
which color a vertex was reached,
and thus, to find $H$-walks that use
some vertex several times.   Once
a pair has joined the queue, it will
be painted black, so it will not be
explored again in the future, just as
in BFS.   But, unlike BFS, when we
are considering an out-neighbor
$y$ of a vertex $x$, we will check
whether $y$ is uncolored, and also
if the color $c(x,y)$ of the arc $(x,y)$
is compatible in $H$ with the color
$c$ we used to reach $x$ (when
$(x,c)$ is the current head of the
queue); in other words, we need that
$(c, c(x,y)) \in A_H$.

Now, we present the pseudo-code
for our algorithm.   The only two
structures that we will use are a
queue $Q$ and a set $R$.   The
input will be an $H$-arc-colored
digraph $D(v)$ with a distinguished
vertex $v$, the ``root'' of our search.
We would like to point out that we
are not trying to optimize the running
time of the algorithm, and probably
there are better implementations of
the same idea we are using, but
we tried to keep the algorithm as
clear and simple as possible.

\vspace{0.5cm}

\begin{algorithm}[H] \label{Hreach}
  \DontPrintSemicolon
  \SetKwData{Left}{left}\SetKwData{This}{this}\SetKwData{Up}{up}
  \SetKwFunction{Union}{Union}\SetKwFunction{FindCompress}{FindCompress}
  \SetKwInOut{Input}{input}\SetKwInOut{Output}{output}

  \KwIn{An $H$-arc-colored digraph $D(v)$.}
  \KwOut{The set $R$ of all vertices reached by $H$-walks from $v$}% and a time function $t$.}
  \BlankLine
  {%$i \leftarrow 0$,
  $Q \leftarrow \varnothing$, $R \leftarrow \varnothing$}\;
  {$R \leftarrow R \cup \{ v \}$}\;
  \For{$c \in V_H$}{
  	%$i \leftarrow i + 1$\;
  	color $(v,c)$ black\;
	%$t(v,c) \leftarrow i$\;
  	append $(v,c)$ to $Q$
  }
  \While{$Q$ is nonempty}{
  	consider the head $(x,c)$ of $Q$\;
 	\eIf{$x$ has an out-neighbor $y$ such that $(y,c(x,y))$ is uncolored and $(c,c(x,y)) \in A(H)$}{
		%$i \leftarrow i +1$\;
		color $(y,c(x,y))$ black\;
		$R \leftarrow R \cup \{ y \}$\;
		%$t(y,c(x,y)) \leftarrow i$\;
		append $(y, c(x,y))$ to $Q$\;
	}
	{remove $(x,c)$ from $Q$}
  }
  {return $R$}%$(R,t)$}
  \caption{Reachability by $H$-walks algorithm}\label{reachHwalks}
\DecMargin{1em}
\end{algorithm}

\begin{thm}
Let $D(v)$ be an $H$-arc-colored digraph.   A
vertex $u \in V_D$ is reached from $v$ by
$H$-walks in $D$ if and only if $u \in R$, where
$R$ is the output of Algorithm \ref{Hreach}
\end{thm}

\begin{proof}
%Let $R'$ be the set of all vertices reached
%from $v$ by $H$-walks in $D$.   We will
%prove by induction on $t$, that $R \subseteq
%R'$.
%
We will use an inductive argument to prove
that every vertex in $R$ is reached from $v$
by $H$-walks in $D$.
Let $u$ be a vertex in $R$.   It follows from
steps $8$ and $9$ of Algorithm \ref{Hreach}
that $(u,c(x,u))$ joined $Q$ while exploring
the pair $(x,c)$, where $x$ is an in-neighbor
of $u$ already in $R$, and such that $(c,
c(x,u)) \in A(H)$.

Assume that, for every pair $(y,c)$ that joined
$Q$ before $(u,c(x,u))$, either $y = v$, or $y$
is reached from $v$ by an $H$-walk, $W$,
such that the color of the last arc of $W$ is $c$. 
Since $(x,c)$ joined $Q$ before $(u,c(x,u))$,
there exists an $H$-walk, $W_x$, from $v$ to
$x$ with the aforementioned property. Clearly
$W_x \cup (x,u)$ is an $H$-walk in $D$ and the
color of its last arc is $c(x,u)$.   The desired
result now follows from the Second Principle
of Mathematical Induction.

It remains to show that every vertex reached
by $H$-walks from $v$ in $D$ belongs to $R$.
Clearly, $v \in R$. Now, let $u$ be a vertex in
$D$ such that $u \ne v$, and an $H$-walk, $W$,
from $v$ to $u$ exists in $D$.   Let $W = ( v = x_0,
\dots, x_n = u)$.   We will prove by induction on
the length of $W$, $\ell (W)$, that $(u,c(x_{n-1},
u))$ is colored black during the running of
Algorithm \ref{Hreach}. If $\ell (W) = 1$, then
$u$ is an out-neighbor of $v$ and the result is
immediate.   Suppose that $\ell (W) = n$.  By the
induction hypothesis, $(x_{n-1}, c(x_{n-2},
x_{n-1}))$ is colored black while Algorithm
\ref{Hreach} is running.   Thus, $(x_{n-1}, c(
x_{n-2}, x_{n-1}))$ is the head of $Q$ at some point,
and the pair $(u, c(x_{n-1}, u))$ is considered
in step $9$.    Since $W$ is an $H$-walk, the
arc $(c(x_{n-2}, x_{n-1}), c(x_{n-1}, u))$ is in
$H$, and hence, $(u,c(x_{n-1},u))$ is explored
(and thus colored black) in this step, unless it
has been previously colored black.   In either
case, the desired pair is colored black, and thus,
it follows from step $11$ in Algorithm \ref{Hreach}
that $u \in R$.
\end{proof}

Let us make a brief running time analysis for
Algorithm \ref{Hreach}.   Steps $1$ and $2$ are
executed a constant number of times.   Steps
$4$ and $5$ are performed once for each
vertex in $H$, this is $o (V_H)$.   For every
arc $(x,y)$, the pair $(y, c(x,y))$ may join $Q$
at most once and, while in $Q$, every arc with
tail $y$ should be considered.   Hence, the
number of times steps $9-15$ are performed
is a linear function of $\sum_{v \in V_D}
d_D^-(v) \cdot d_D^+(v) \le \sum_{v \in V_D}
\Delta_D^- \cdot d_D^+(v) = \Delta_D^- \cdot
|A_D| \le |V_D| \cdot |A_D|$.   Hence, the
running time of Algorithm \ref{Hreach} is
$o(|V_H| + |V_D| \cdot |A_D|)$, which is
polynomial.

Moreover, Algorithm \ref{Hreach} can be
modified to omit the $|V_H|$ from the running
time.   Just add directly $(x,c(v,x))$ to $Q$ for
every out-neighbor of $v$ instead of
performing steps $3-6$; this is done at most
$|V_D|$ times.   Thus, the running time of this
modified version of Algorithm \ref{Hreach} is
$o (|V_D| \cdot |A_D|)$.

As we have discussed at the beginning of
this section, the following result is already
proved.

\begin{cor}
The problem of determining whether an
$H$-arc-colored digraph has a kernel by
$H$-walks is in $NP$.
\end{cor}

\section{Minimal obstructions for panchromaticity}
\label{sec:minobs}

Given an $m$ by $m$ matrix $M$ over $\{ 0, 1, * \}$, an
$M$-partition of a digraph $D$ is a partition of its vertex set
into parts $V_1, \dots, V_m$ such that each vertex in $V_i$
must (respectively must not) dominate each vertex in $V_j$
if $M_{i,j}=1$ (respectively $M_{i,j}=0$); there are no
restrictions if $M_{i,j} = \ast$.   When $i = j$, $V_i$ is a strong
clique (respectively an independent set) if $M_{i,i} = 1$
(respectively  $M_{i,i} = 0$).   If $M$ is a symmetric matrix, a
similar definition can be given for undirected graphs.   An
excellent survey on the subject of matrix partitions of graphs
and digraphs is due to Hell \cite{survey}.

The $M$-partition problem asks whether a given digraph $D$
admits an $M$-partition.   Particular choices of matrices $M$
yield very well known problems in the undirected graph case,
e.g.,
a $\left(
\begin{array}{cc}
0 & \ast \\
\ast & 0
\end{array}
\right)$-partition is simply a $2$-coloring and a
$\left(
\begin{array}{cc}
0 & \ast \\
\ast & 1
\end{array}
\right)$-partitionable graph is a split graph.

It is easy to observe that having an $M$-partition is a hereditary
property, and hence, it can be characterized through a set of
minimal forbidden induced subdigraphs.   For a matrix $M$, we
define an $M$-obstruction to be any digraph not having an
$M$-partition.   An $M$-obstruction, $D$, is minimal if any
induced proper subdigraph of $D$ has an $M$-partition.
Of course, it is direct to verify that a digraph admits an
$M$-partition if and only if it does not contain any minimal
$M$-obstruction as an induced subdigraph.

Let $M_1$ and $M_2$ be the $2 \times 2$ matrices
$$M_1 =
\left(
\begin{array}{cc}
1 & 1 \\
\ast & 1
\end{array}
\right)
\hspace{2cm}
M_2 =
\left(
\begin{array}{cc}
1 & 0 \\
0 & 1
\end{array}
\right).
$$
Based on the work of Arpin and Linek \cite{AL},
Galeana-S\'anchez and Strausz proved that a digraph
is a panchromatic pattern if and only if it is a looped
$M_1$-partitionable or $M_2$-partitionable digraph,
\cite{GS-S}.   Again, it is easy to observe that having
an $M_1$-partition or an $M_2$-partition is a hereditary
property.   We define a {\em panchromatic obstruction}
to be a digraph having neither an $M_1$-partition nor
an $M_2$-partition. A panchromatic obstruction is {\em
minimal} if every induced subdigraph admits an
$M_1$-partition or an $M_2$-partition.   The aim of this
section is to characterize all the minimal panchromatic
obstructions.   To this end, we will use the characterization
given by Hell and Hern\'andez-Cruz in \cite{HHC} of all the
$M_1$-minimal obstructions and $M_2$-minimal
obstructions.

\begin{figure}
\begin{center}
\includegraphics[width=\textwidth]{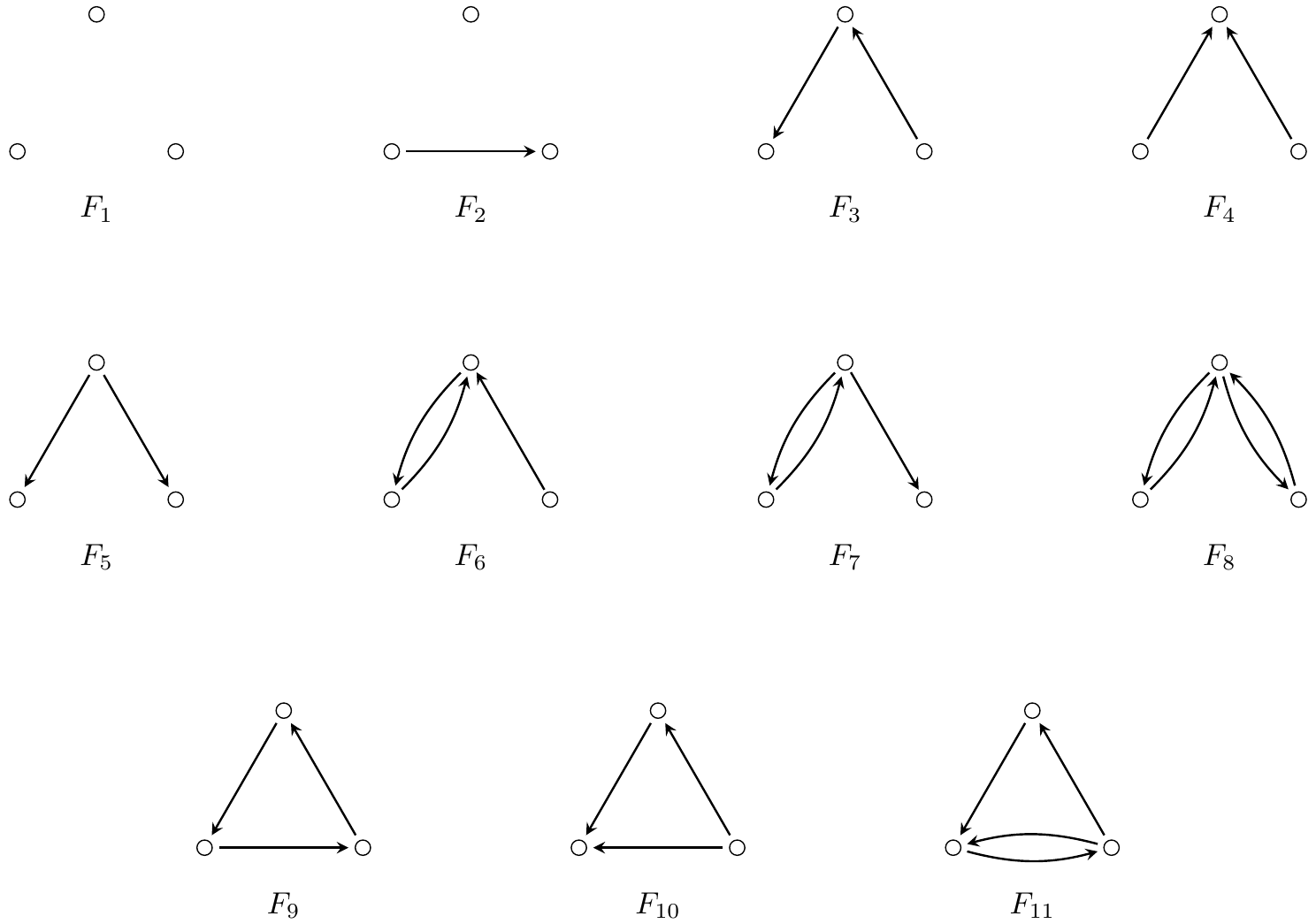}
\caption{Minimal panchromatic obstructions.} \label{M1M2MinObs}
\end{center}
\end{figure}

\begin{thm}
Let $M_1$ and $M_2$ be the $2 \times 2$ matrices defined
above.   Then,
\begin{enumerate}
	\item The minimal $M_1$-obstructions are, an
			independent set of two vertices, and the
			digraphs $F_9, F_{10}$ and $F_{11}$
			in Figure \ref{M1M2MinObs}.
	
	\item The minimal $M_2$-obstructions are, an
			asymmetric arc, and the digraphs $F_1$ and
			$F_8$ in Figure \ref{M1M2MinObs}.
\end{enumerate}
\end{thm}

Let $\mathcal{F}$ be the set $\{F_i\}_{i=1}^{11}$,
depicted in Figure \ref{M1M2MinObs}.
Observe that the digraphs in $\mathcal{F}$ are
minimal panchromatic obstructions.   Indeed, the
digraphs $F_9, F_{10}$ and $F_{11}$ are
minimal $M_2$-obstructions containing at least
one asymmetric arc, which makes them also
$M_2$-obstructions.   Similarly, $F_1$ is a
minimal $M_2$-obstruction containing an
independent set of size $2$, which makes it an
$M_1$-obstruction.   Each of the remaining
digraphs in $\mathcal{F}$ contains an
asymmetric arc and an independent set of
cardinality $2$, and hence is both an
$M_1$-obstruction and an $M_2$-obstruction.
These obstructions are minimal since
any digraph on $2$ vertices is either
$M_1$-partitionable or $M_2$-partitionable.
Hence, every digraph in $\mathcal{F}$ is a
minimal panchromatic obstruction.   Our next
result shows that these are the only minimal
panchromatic obstructions.

\begin{thm} \label{panchar}
If $D$ is a digraph, then $D$ admits an $M_1$-partition
or an $M_2$-partition if and only if it is $\mathcal{F}$-free.
\end{thm}

\begin{proof}
For the necessity, we have already observed in the
previous paragraph that the digraphs in $\mathcal{F}$
are minimal panchromatic obstructions.   Therefore,
a digraph $D$ admitting an $M_1$-obstruction or an
$M_2$-obstruction cannot contain any of the digraphs
in $\mathcal{F}$.

For the sufficiency, we
will proceed by contrapositive.   Let
$D$ be a panchromatic obstruction.

The minimal $M_1$-obstructions and the minimal
$M_2$-obstructions on $3$ vertices are in the set
$\mathcal{F}$.   Thus, if $D$ contains any of them,
we are done.   Otherwise, in order for $D$ to be a
panchromatic obstruction, $D$ should contain the
only minimal $M_1$-obstruction on two vertices, (an
independent set of size two), and the only minimal
$M_2$-obstruction on two vertices, (an asymmetric
arc).

Let $\{v_1, v_2\}$ be an independent set, and let
$(v_3, v_4)$ be an asymmetric arc of $D$.   First
suppose that $\{v_1, v_2\} \cap \{v_3, v_4\} \ne
\varnothing$.   This intersection contains
precisely one vertex. Hence, the set $\{v_1, v_2\}
\cup \{ v_3, v_4 \}$ has cardinality $3$, and it is
easy to verify that it induces the digraph $F_i$ for
some $i \in \{2,3,4,5,6,7\}$. Thus, we are done.

Else,  $\{v_1, v_2\} \cap \{v_3, v_4\} =
\varnothing$.   If there are no arcs, or there
is an asymmetric arc, between $\{v_1, v_2\}$ and
$\{v_3, v_4\}$, then we can choose an asymmetric
arc with an end in $\{v_1, v_2\}$, or we can
choose an independent set of size two sharing a
vertex with $(v_3, v_4)$, as in the previous
case.   Hence, every arc between $\{v_1, v_2\}$
and $\{v_3, v_4\}$ is present, and is a digon.
In this case, the set $\{v_1, v_2, v_3 \}$
induces the digraph $F_8$.

Since the cases are exhaustive and in all of them
we obtain the existence of an induced subdigraph
of $D$ isomorphic to a member of $\mathcal{F}$,
the proof is complete.
\end{proof}

\begin{cor}
Panchromatic patterns can be recognized in
polynomial time.
\end{cor}

\begin{proof}
To verify if a digraph $D$ is a panchromatic
pattern, it suffices to check whether $D$ has
induced subdigraphs isomorphic to some
member of $\mathcal{F}$.   Using a brute
force approach, this amounts to verify every
subset of $3$ vertices of $V_D$, which can
be done in roughly $o(|V_D|^3)$-time.
\end{proof}

\section{The reduction for $F_1, F_5, F_7$ and $F_8$}
\label{sec:genreduc}

Before proving our main result, we will
construct a polynomial reduction scheme
that will be used to prove that any looped
digraph $H$ containing one of $F_1, F_5,
F_7$ or $F_8$ as an induced subdigraph,
has an $NP$-complete kernel by $H$-walks
problem. We will reduce it from the
$k$-colouring problem for graphs
($k$-\textsc{Col}). We would like to point
out that we will not construct a single
polynomial reduction, but a family of
reductions that, depending on a part of
the vertex gadget, will work for different
choices of $H$.

For the remainder of this section $H$
will denote a looped digraph containing
$F_1, F_5, F_7$ or $F_8$ as an induced
subdigraph.   For a given graph $G$, we
will construct a digraph $D_G$ such that,
the graph $G$ is $k$-colorable if and only
if $D_G$ has a kernel by $H$-walks.   In
the following construction, {\em red, green}
and {\em blue} will be the vertices of the
induced copy of $F_i$, $i \in\{1, 5, 7, 8\}$,
in $H$; the set $\{ green, blue \}$ will be
an independent set of $H$.

Let $G$ be a graph, and consider any of
its acyclic orientations, $\overrightarrow{G}$.
Let $k \ge 3$ be a fixed integer.   For each
vertex $v$ of $G$, construct a $k$-cycle
$C_v = (x_{v1}, \dots, x_{vk}, x_{v1})$, and
color each of its arcs green.   Also, create a
copy $F_v$ of an $H$-arc-colored digraph
$F$ not having a kernel by $H$-walks, add
all the arcs from $F_v$ to $C_v$, and color
them green.   The resulting colored digraph
is the gadget for the vertex $v$.    Observe
that this gadget depends on the choice
of $F$.

Now, for every arc $(u,v)$ of
$\overrightarrow{G}$, and for every $1 \le i
\le k$, create a directed $4$-cycle
$Q_{(u,v)i}$ with arcs $(x_{(u,v)i}, y_{(u,v)i})$
and $(z_{(u,v)i}, w_{(u,v)i})$, colored green,
and $(y_{(u,v)i}, z_{(u,v)i})$ and $(w_{(u,v)i},
x_{(u,v)i})$ colored blue.   Finally, add the
arcs $(x_{ui}, x_{(u,v)i})$ and $(w_{(u,v)i},
x_{vi})$ colored blue. Let $D_G$ be the
resulting colored digraph.   Notice that
the $4$-cycle $Q_{(u,v)i}$ could be
replaced by any even cycle colored with
alternating colors, or with a blue
colored digon.   Nonetheless, we prefer
to keep the $4$-cycle because, if $F$
does not have digons, then $D_G$
neither has digons. Figure \ref{arcgadget}
shows the construction of the gadget for
the arc $(u,v)$ when $k = 3$.   The dashed
lines represent green arcs and the solid
lines represent blue arcs.

\begin{figure}
\begin{center}
\includegraphics{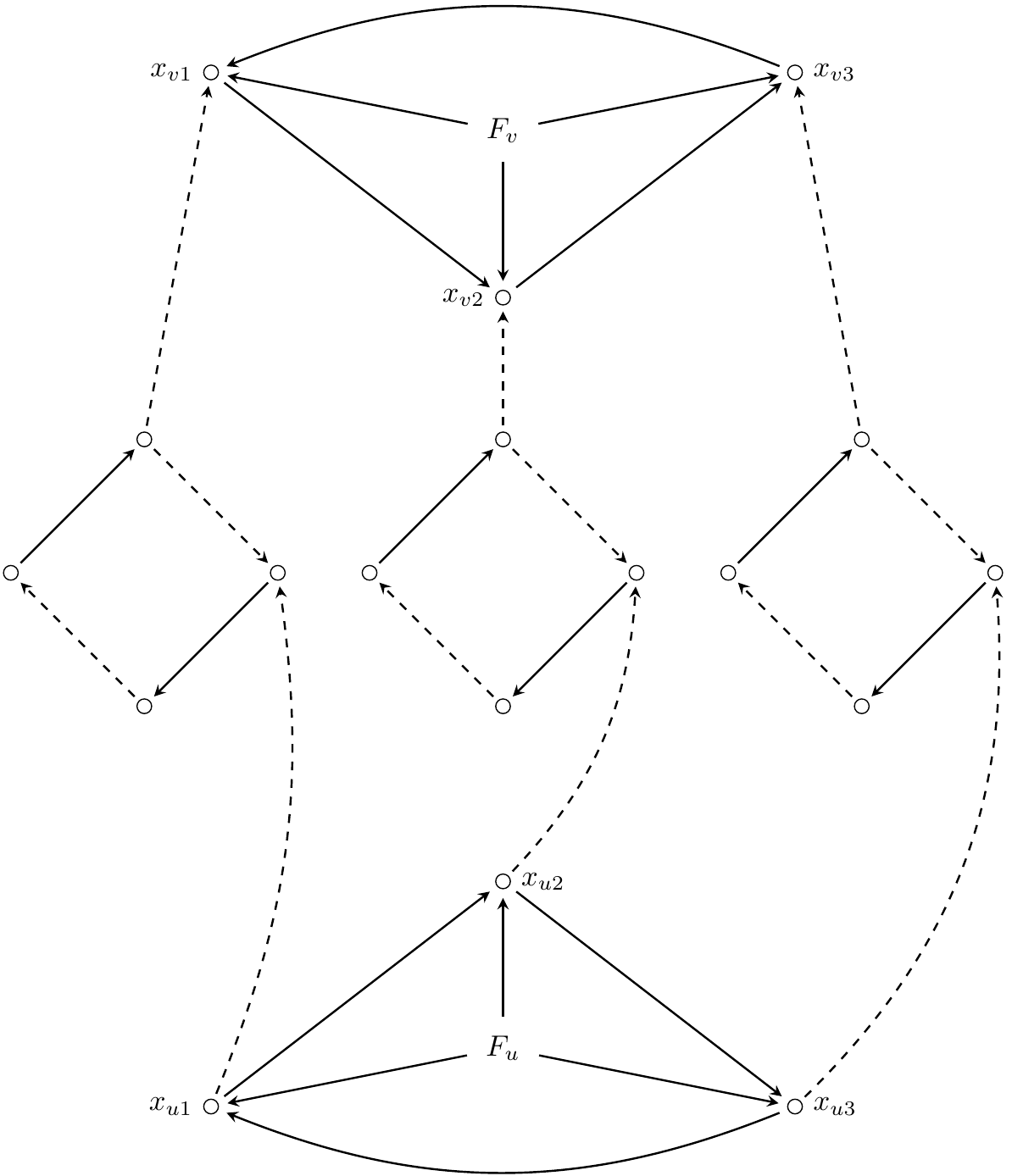}
\caption[width=\textwidth]{Gadget for the arc $(u,v)$} \label{arcgadget}
\end{center}
\end{figure}

\begin{lem} \label{gadgetdifcol}
Let $G$ be a graph.   If $D_G$ has a kernel by $H$-walks, $K$, then,
\begin{enumerate}
	\item For every vertex $v$ of $G$, precisely one vertex of $C_v$
		belongs to $K$.
	 \item If $(u,v)$ is an arc of $\overrightarrow{G}$, then,
	 	\begin{enumerate}
			\item If $x_{ui} \in K$, then $x_{vi} \notin K$.
			\item If $x_{vi} \in K$, then $x_{ui} \notin K$.
		\end{enumerate}
\end{enumerate}
\end{lem}

\begin{proof}
Let $v$ be any vertex of $G$.   Since $F_v$ does not have a kernel
by $H$-walks,  then at least one vertex of $F_v$ must be absorbed
by a vertex in $K \setminus V(F_v)$.    By construction, $\partial^+
(F_v) = V(C_v)$, and all the arcs going out from $C_v$ are blue,
hence, $V(C_v) \cap K \ne \varnothing$.   On the other hand,
$C_v$ is a monochromatic cycle, and thus $|K \cap V(C_v)| \le 1$.
Therefore, $|K \cap V(C_v)| = 1$.

Now, suppose that $x_{ui} \in K$.   Since $(x_{ui}, x_{(u,v)i})$ is an
arc of $D$, we have $x_{(u,v)i} \notin K$.   But $Q_{(u,v)i}$ is a
$4$-cycle with alternating colors blue and green, hence, locally,
$K$ behaves as a regular kernel.   Thus, $y_{(u,v)i}$ and
$w_{(u,v)i}$ must belong to $K$ in order to absorb $x_{(u,v)i}$
and $z_{(u,v)i}$ by $H$-walks, respectively.   Since $(w_{(u,v)i},
x_{vi})$ is an arc of $D_G$, we have  $x_{vi} \notin K$.

Similarly, if $x_{vi} \in K$, then $w_{(u,v)i}$ is absorbed by
$H$-walks by $x_{vi}$, and hence, it cannot belong to $K$.
This fact, together with the structure of $Q_{(u,v)i}$ implies
that $z_{(u,v)i}$ and $x_{(u,v)i}$ belong to $K$.   Since
$(x_{ui}, x_{(u,v)i})$ is an arc of $D_G$, we have $x_{ui}
\notin K$.
\end{proof}

\begin{lem} \label{gadgetunique}
Let $G$ be a graph, and let $I$ be an independent by $H$-walks
subset of $V(D_G)$ such that $I \cap V(C_v) \ne \varnothing$ for
every $v \in V(G)$.   Let $(u,v)$ be an arc of $\overrightarrow{G}$,
and let $I \cap V(C_u) = \{ x_{ui} \}$ and $I \cap V(C_v) = \{ x_{vj}
\}$.   If for every arc $(u,v)$ of $\overrightarrow{G}$ we have $i
\ne j$, then $I$ can be extended to be a kernel by $H$-walks of
$D_G$.
\end{lem}

\begin{proof}
Consider $I$ as in the hypothesis.   By construction, every vertex
in the gadget for every $v \in V(G)$ is absorbed by the only vertex
in $I \cap V(C_v)$.

Let $(u,v)$ be an arc in $\overrightarrow{G}$ and suppose that $I
\cap V(C_u) = \{ x_{ui} \}$.   For every $j \ne i$, add $x_{(u,v)j}$
and $z_{(u,v)j}$ to $I$; also, add $y_{(u,v)i}$ and $w_{(u,v)i}$ to
$I$.   Repeat this process with every arc of $\overrightarrow{G}$
and let $K$ be the resulting subset of $V(D_G)$.   In the former
case, $y_{(u,v)j}$ and $w_{(u,v)j}$ are absorbed by $H$-walks
by $z_{(u,v)j}$ and $x_{(u,v)j}$, respectively.   In the latter case,
$x_{(u,v)i}$ and $z_{(u,v)i}$ are absorbed by $H$-walks by
$y_{(u,v)i}$ and $w_{(u,v)i}$, respectively.

For the independence by $H$-walks of $K$, first observe that,
except from the arcs from $F_v$ to $C_v$, every arc coming from
and going to $C_v$ in $D_G$ is blue.   Since all the arcs from
$F_v$ to $C_v$, and all the arcs in $C_v$ are green, the set
$\{ x_{ui}, x_{vj} \}$ is independent by $H$-walks, for every pair
of different vertices $u, v \in V(G)$.   Recalling the construction
of $K$ in the preceding paragraph, $x_{(u,v)i} \notin K$ when
$x_{ui} \in K$.   Also, this is the only case where $w_{(u,v)i} \in
K$, but by the choice of $I$, we have $x_{vi} \notin K$.   Hence
$K$ is independent by $H$-walks.
\end{proof}

\begin{lem} \label{1578red}
If $G$ is a graph, then $G$ is $k$-colorable if and only if $D_G$
has a kernel by $H$-walks.
\end{lem}

\begin{proof}
Let $c: V(G) \to \{ 1, \dots, k \}$ be a $k$-coloring of $G$.   Consider
the subset $I \subset V(D_G)$ defined as $\{ x_{vi} \colon\  v \in V(G),
c(v) = i \}$.   Clearly, $I$ is a set fulfilling the conditions of Lemma
\ref{gadgetunique}, so it can be extended in an unique way to a
kernel by $H$-walks, $K$, of $D_G$.

Conversely, let $K$ be a kernel by $H$-walks of $D_G$.   By Lemma
\ref{gadgetdifcol}, for every vertex $v \in V(G)$, there is precisely one
vertex in $K \cap V(C_v)$, say, $x_{vi}$.    If we define a $k$-coloring,
$c$, of $G$ as $c(v) = i$ for every $v \in V(G)$, then Lemma
\ref{gadgetdifcol} implies that adjacent vertices receive different
colors.   Thus, $c$ is a proper $k$-coloring of $G$.
\end{proof}

\section{Main Results} \label{sec:mainres}
 
We have already proved in Section
\ref{sec:algorithm} that the kernel by $H$-walks
problem is in $NP$.   Hence, in order to prove
that it is $NP$-complete for a particular choice
of $H$, it suffices to prove its $NP$-hardness.
As usual, we will achieve this through a
polynomial reduction.   To prove the dichotomy
we claim, we must show that every digraph $H$
which is not a panchromatic pattern, has an
$NP$-hard kernel by $H$-walks problem.
Since every panchromatic pattern is a looped
digraph, we will begin discussing the digraphs
$H$ which are missing at least one loop.

Let $H$ be a digraph, and let $red$ be a vertex
of $H$ without loops.   Clearly, if we consider
any digraph $D$, and color each of its arcs red
to obtain the $H$-arc-colored digraph $D'$, then
the only $H$-walks of $D'$ are its arcs.  Thus, a
set $K \subseteq V_D$ is a kernel of $D$ if and
only if it is a kernel by $H$-walks of $D'$.   This
argument describes a linear reduction of the
kernel by $H$-walks problem from the kernel
problem.   Moreover, as discussed on Section
\ref{sec:Intro}, $D$ can be chosen to be
$3$-colorable, or planar with $\Delta^-,
\Delta^+ \le 2$ and $\Delta \le 3$, and the
problem remains $NP$-complete.   So, we
have proven the following result.
 
\begin{prop} \label{loopless}
Let $H$ be a digraph with at least one loopless
vertex.   The problem of determining whether an
arc-colored digraph $D$ has a kernel by
$H$-walks is $NP$-complete, even when
restricted to $3$-colorable digraphs, or to planar
digraphs with $\Delta^-, \Delta^+ \le 2$ and
$\Delta \le 3$.
 \end{prop}

Now, we can consider only looped digraphs $H$.
Although the forbidden induced subdigraph
characterization given in Section \ref{sec:minobs}
deals with loopless digraphs, it is clear that the
same characterization works with looped
digraphs, we just have to consider the looped
version of each digraph in the family $\mathcal{F}$.
Our next result deals with the digraphs $H$
containing a member of the family $\mathcal{F}$,
other than $F_1, F_5, F_7$ or $F_8$, as an
induced subdigraph.
 
\begin{thm} \label{mostcases}
Let $H$ be a looped digraph containing three
vertices, $red$, $blue$ and $green$, such that
there is an asymmetric arc from red to green,
and the arc from red to blue is missing.   The
problem of determining whether an
$H$-arc-colored digraph $D$ has a kernel by
$H$-walks is $NP$-complete, even when
restricted to planar bipartite digraphs with
$\Delta^-, \Delta^+ \le 2$, and $\Delta \le 3$.
\end{thm}

\begin{proof}
We will reduce it from the kernel problem,
which is known to be $NP$-complete even
when restricted to planar digraphs with
$\Delta^-, \Delta^+ \le 2$, and $\Delta \le 3$.

Let $D$ be a planar digraph with the
aforementioned degree restrictions.   We will
construct a new $H$-arc-colored digraph
$D'$ in the following way. Subdivide every
arc $(x,y)$ of $D$ by adding the intermediate
vertex $v_{(x,y)}$, and create an additional
vertex $v'_{(x,y)}$ along with the arc
$(v_{(x,y)}, v'_{(x,y)})$.   For every arc $(x,y)$
of $D$, color red, green and blue the arcs
$(x, v_{(x,y)})$, $(v_{(x,y)}, y)$ and $(v_{(x,y)},
v'_{(x,y)})$ of $D'$, respectively.   Clearly,
the digraph $D'$ is planar, bipartite, and
have the desired degree constraints.
We will prove that $D$ has a kernel if and
only if $D'$ has a kernel by $H$-walks.

Let $S$ and $T$ be the subsets of $V_D'$
defined as $$S = \{ v'_{(x,y)} \colon\ (x,y) \in
A_D \},$$ $$T = \{ v_{(x,y)} \colon\ (x,y) \in
A_D \}.$$    Since every vertex in $S$ has
out-degree equal to zero, $S$ must be
contained in every kernel by $H$-walks
of $D'$.   Moreover, every vertex in $T$
is absorbed by some vertex in $S$, so
$T$ does not intersect any kernel by
$H$-walks of $D$.

\begin{claim1} \label{C11}
For every vertex $x \in V_D \cap V_{D'}$,
the vertices reached from $x$ in $D'$ by
$H$-walks are precisely the vertices in
the set $$R = N_D^+(x) \cup \{ v_{(x,y)}
\colon\ y \in N^+(x) \}.$$
\end{claim1}

Suppose first that $D$ has a kernel $K$.
We affirm that $K' = K \cup S$ is a kernel
by $H$-walks for $D'$.   We have already
observed that every vertex of $T$ is
absorbed by $H$-walks by $K'$.   Let $x$
be a vertex in $V_{D'} \cap V_D$.   If $x
\notin K$, then there exists $y \in K$ such
that $(x,y) \in A_D$.   It follows from Claim
\ref{C11} that $x$ is absorbed by $H$-walks
by $y$.  Hence $K'$ is absorbent by
$H$-walks in $D'$.   For the independence
by $H$-walks of $K'$, notice that the
vertices in $S$ have zero out-degree, so
they cannot reach any other vertex. Now,
for the vertices in $K$, it follows from
Claim \ref{C11} that they cannot reach
the vertices in $S$ by $H$-walks, and
the only vertices in $V_D \cap V_{D'}$ they
can reach by $H$-walks are precisely
the vertices in $N_D^+(x)$.   Thus, $K'$
is independent by $H$-walks.

Now, suppose that $K'$ is a kernel by
$H$-walks of $D'$.   We affirm that $K =
K' \setminus S$ is a kernel of $D$.   Let
$x$ be a vertex in $V_D \setminus K$.
Since $T \cap K' = \varnothing$, it follows
from Claim \ref{C11} that there exists a
vertex $y \in N_D^+ (x)$ such that $y \in
K'$.   Therefore, $y \in K$ and $K$ is
absorbent.   Claim \ref{C11} together
with the independence by $H$-walks of
$K'$ imply that $K$ is an independent set
of $D$.

\begin{proof}[Proof of Claim \ref{C11}]
Since there is an
arc from red to green in $H$, for any arc
$(x,y) \in A_D$, the walk $(x, x_{(x,y)}, y)$
is an $H$-walk in $D'$.   Since the arcs
from red to blue and from green to red
are missing in $H$, the walk $(x, x_{(x,y)},
y)$ is a maximal $H$-walk in $D'$, and
the walk $(x, v_{(x,y)}, v'_{(x,y)})$ is not
an $H$-walk.
\end{proof}
\end{proof}

Finally, we deal with digraphs $H$
containing one of $F_1, F_5, F_7$ or
$F_8$.

\begin{thm} \label{1578NPc}
Let $H$ be a digraph containing $F_i$
as an induced subdigraph for some $i \in
\{ 1, 5, 7, 8 \}$.   The kernel by $H$-walks
problem is $NP$-complete.
\end{thm}

\begin{proof}
Lemma \ref{1578red} shows that the kernel
by $H$-walks problem can be polynomially
reduced from the $k$-coloring problem for
graphs, for any $k \ge 3$.
\end{proof}

We are now ready to prove our main result.

\begin{proof}[Proof of Theorem \ref{mainthm}]
Let $H$ be a digraph which is not a
panchromatic pattern.   If $H$ is not a looped
digraph, Proposition \ref{loopless} implies that
the kernel by $H$-walks problem is
$NP$-complete.   Else, Theorem \ref{panchar}
implies that $H$ contains an element of
$\mathcal{F}$ as an induced subdigraph.
It follows from Theorems \ref{mostcases} and
\ref{1578NPc} that the kernel by $H$-walks
problem is $NP$-complete.
\end{proof}

\section{Concluding remarks} \label{sec:conclusions}

We find interesting that, although for every
digraph $H$ containing $F_7$ or $F_8$ as
an induced subdigraph, there exists an
arc-colored digraph $D$ such that $D$ does
not have a kernel by $H$-walks, no examples
are known of such $D$.   We think that once
these examples are found, Theorem \ref{1578NPc}
could be improved by restricting the family of
digraphs $D$ where the kernel by $H$-walks
problem remains $NP$-complete.   This
improvement can be achieved for digraphs $H$
containing $F_1$ or $F_5$ as an induced
subdigraph.   Figure \ref{15obs} shows an
$F_5$-arc-colored digraph without a kernel by
$F_5$-walks (on the left, Arpin and Linek \cite{AL})
and an $F_1$-arc-colored digraph without a
kernel by $F_1$-walks (on the right).   In both
cases, the doubled arcs correspond to the top
vertex of $F_i$, $i \in \{1, 5\}$, in Figure
\ref{M1M2MinObs}.

\begin{figure}
\begin{center}
\includegraphics{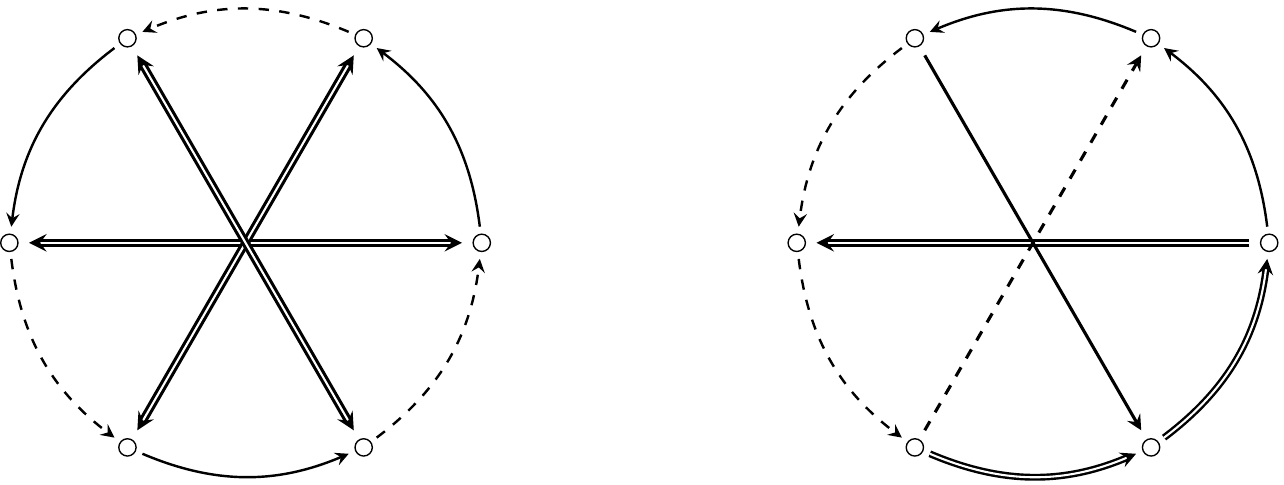}
\caption{On the left, an $F_5$-arc-colored digraph without a kernel by $F_5$-walks.
	On the right, a bipartite tournament without a kernel by monochromatic paths.}
	\label{15obs}
\end{center}
\end{figure}

\begin{cor} \label{cor15}
Let $H$ be a digraph containing $F_1$
or $F_5$ as an induced subdigraph. The kernel
by $H$-walks problem is $NP$-complete, even
when restricted to bipartite digraphs with
circumference $k \ge 6$.
\end{cor}

\begin{proof}[Sketch of proof.]
Notice that we can choose the digraph $F_v$
used in the gadget for the vertex $v$ of the
reduction given in Section \ref{sec:genreduc}
to be one of the digraphs in Figure \ref{15obs},
let us call it $F$.   Since both digraphs are
bipartite, if we choose any even integer $k
\ge 6$ in the reduction given in Section
\ref{sec:genreduc}, then we can modify the
construction of the vertex gadget to obtain
a bipartite digraph (we do not need all the
arcs from $F$ to the $k$-cycle, so we can
choose them accordingly to the bipartition
of $F$ and the $k$-cycle).

Notice that the gadget for every arc is a
bipartite digraph, so the resulting digraph
of the construction $D'$ is a bipartite digraph.
Clearly, the circumference of $D'$ is $k$, so
we have our desired result.
\end{proof}

A result analogous for $F_7$ or $F_8$
would be obtained if an $F_i$-arc-colored
bipartite digraph without a kernel by
$F_i$-walks could be found, $i \in \{ 7, 8 \}$.
Moreover, the construction of the digraph
$D'$ in Section \ref{sec:genreduc} is quite
flexible, so results analogous to Corollary
\ref{cor15} with some other restrictions on
the input could be obtained depending
on the properties of those digraphs.   We
think that the following simple problem is
interesting.

\begin{problem}
Find an $F_i$-arc-colored (bipartite) digraph
without a kernel by $F_i$-walks, $i \in \{7, 8\}$.
\end{problem}

It is also worth noticing that having more
information on the structure of $H$ might
yield results restricting even more the
structure of the input digraph.   Consider
the following result.

\begin{thm}
It is $NP$-complete to determine whether
a digraph with $4$-colored arcs has a
kernel by monochromatic paths, even
when restricted to planar digraphs with
$\Delta^+, \Delta^- \le 2$ and $\Delta \le 3$.
\end{thm}

\begin{proof}[Sketch of proof.]
The reduction is from the kernel problem
in planar graphs with the required degree
constraints.   Let $D$ be an input digraph
for the kernel problem.  By Vizing's
Theorem we can find a proper $4$-coloring
of the arcs of $D$ to obtain a digraph
with $4$-colored set of arcs $D'$.
Clearly, a subset $S$ of $V_D$ is
a kernel of $D$ if and only if it is a
kernel by monochromatic paths of
$D'$.
\end{proof}

Now that the complexity of the kernel by
$H$-walks problem has been completely
classified, what else can be done in order
to find sufficient conditions for the existence
of kernels by $H$-walks?    Of course it is
useless to restrict the digraph $H$, and our
previous remarks on the flexibility on the
construction given in Section
\ref{sec:genreduc} intuitively say that
restricting the structure of the input
digraph $D$ will not be very useful either.
We think that it would be interesting to
restrict the way the arcs of $D$ are colored
with $V_H$; this idea has been previously
explored in \cite{GS3} and it may be worth
considering again.

Probably, the most interesting problem
we can think of right now is the following.

\begin{problem}
Is there a dichotomy for the kernel by
$H$-paths problem?
\end{problem}

It is not hard to verify that most of our
reductions also work for the kernel by
$H$-paths problem. But of course, we
do not have a panchromatic pattern
characterization for the path version
of the problem.   So that would be a
good starting point.


\begin{thebibliography}{50}
\bibitem{AL}
	P.~Arpin and V.~Linek,
	Reachability problems in edge-colored digraphs,
	Discrete Math. 307 (2007) 2276--2289.

\bibitem{BJD}
	J.~Bang-Jensen and G.~Gutin,
	Digraphs.  Theory, Algorithms and Applications.   Springer-Verlag, 2002.

\bibitem{BJGGV}
	J.~Bang-Jensen, Y.~Guo, G.~Gutin and L. Volkmann,
	A classification of locally semicomplete digraphs,
	Discrete Math. 167/168 (1997) 101--114.

\bibitem{BM}
	J.A.~Bondy and U.S.R.~Murty,
	Graph Theory,
	Springer-Verlag (2008).

\bibitem{C}
	V. Chv\'atal,
	On the computational complexity of finding a kernel,
	Technical Report Centre de recherches math\' ematiques, Universit\' e
	de Montr\' eal, CRM-300, 1973.

\bibitem{F}
	A.S.~Fraenkel,
	Planar kernel and Grundy with $d \le 3, d^+ \le 2, d^- \le 2$ are NP-complete,
	Discrete Applied Mathematics 3 (1981) 257--262.

\bibitem{GS}
	H.~Galeana-S\'anchez,
	On monochromatic paths and monochromatic cycles in edge colored tournaments,
	Discrete Math. 156 (1996) 103--112.

\bibitem{GS2}
	H.~Galeana-S\'anchez,
	Kernels in edge colored digraphs,
	Discrete Math. 184 (1998) 87--99.

\bibitem{GS3}
	H.~Galeana-S\'anchez,
	Kernels by monochromatic paths and the color-class digraph,
	Discussiones Mathematicae Graph Theory 31(2) (2011) 273--281.

\bibitem{GS-L-MB}
	H.~Galeana-S\'anchez, B.~Llano and J.J.~Montellano-Ballesteros,
	Kernels by monochromatic paths in $m$-colored unions of quasi-transitive digraphs,
	Discrete Applied Mathematics 158 (2010) 461--466.

\bibitem{GS-S}
	H.~Galeana-S\'anchez and R. Strausz,
	On Panchromatic Patterns,
	Graphs and Combinatorics, Accepted.

\bibitem{HIW}
	G.~Hanh, P.~Ille and R.~Woodrow,
	Absorbing sets in arc-colored tournaments,
	Discrete Math. 283 (2004) 93--99.

\bibitem{survey}
	P.~Hell,
	Graph partitions with prescribed patterns,
	European Journal of Combinatorics 35 (2014) 335--353.

\bibitem{HHC1}
	P.~Hell and C.~Hern\'andez-Cruz,
	On the complexity of the $3$-kernel problem in some classes of digraphs,
	Discussiones Mathematicae Graph Theory 34(1) (2014) 167--185.
	
\bibitem{HHC}
	P.~Hell and C.~Hern\'andez-Cruz,
	Minimal digraph obstructions for small matrices,
	arXiv (2016).
	
\bibitem{LS}
	V.~Linek and B.~Sands,
	A note on paths in edge-colored tournaments,
	Ars Combin. 44 (1996) 225--228.
	
\bibitem{SSW}
	B.~Sands, N.~Sauer and R.~Woodrow,
	On Monochromatic Paths in Edge-Coloured Digraphs,
	Journal of Combinatorial Theory, Series B, 33 (1982) 271--275.

\bibitem{S}
	M.G.~Shen,
	On monochromatic paths in $m$-colored tournaments,
	J. Combin. Theory Ser. B 45(1) (1998) 108--111.
\end{thebibliography}
\end{document}